\documentclass[12pt]{article}
\usepackage{amsmath}
\usepackage{esint}
\usepackage[all]{xy}
\usepackage[T1]{fontenc}
\usepackage{blindtext}
\usepackage{hyperref}
\usepackage{pstricks}    
\usepackage{graphicx}    
\usepackage{a4wide}
\usepackage{amssymb}
\usepackage{amsthm}
\usepackage{amscd}
\usepackage{color}

\theoremstyle{plain}
\begingroup
\newtheorem{theorem}{Theorem}[section]
\newtheorem{theorem*}{Theorem*}[section]

\newtheorem{proposition}[theorem]{Proposition}

\newtheorem{definition}[theorem]{Definition}

\endgroup

\makeindex 

\numberwithin{equation}{section}

\newcommand{\op}[1]{{\text{#1}}}

\newcommand{\mb}[1]{\mathbb{#1}}

\begin{document}
\title{A singular radial connection over $\mathbb B^5$ minimizing the Yang-Mills energy}
\author{Mircea Petrache\footnote{Forschungsinstitut f\"ur Mathematik, ETH Zentrum, CH-8093 Z\"urich, Switzerland.}}
\maketitle

\begin{abstract}
 We prove that the pullback of the $SU(n)$-soliton of Chern class $c_2=1$ over $\mathbb S^4$ via the radial projection $\pi:\mathbb B^5\setminus\{0\}\to\mathbb S^4$ minimizes the Yang-Mills energy under the fixed boundary trace constraint. In particular this shows that stationary Yang-Mills connections in high dimension can have singular sets of codimension $5$. 
\end{abstract}

\section{Introduction}
Let $G$ be acompact connected Lie group and $E\to M$ be a vector bundle associated to the adjoint representation of a principal $G$-bundle $P\to M$ over a compact Riemannian $n$-manifold $M$. Following \cite{DoKr} let $\nabla=d+A$ locally represent a connection over $E$ in a given trivialization and let the Lie algebra valued $2$-form representing the curvature of $\nabla$ be given by $F=dA+A\wedge A$. We recall that the Yang-Mills functional is defined in term of the $\op{ad}$-invariant norm $|\cdot|$ on $\mathfrak g$ by 
\[
 \mathcal{YM}(\nabla)=\int_M|F|^2d\op{vol}_M\ .
\]
Consider now the case $n=4, M=\mathbb S^4$ with the standard metric and $G=SU(n)$. We denote by
\[
 F_{\mathbb S^4}=dA_{\mathbb S^4} + A_{\mathbb S^4}\wedge A_{\mathbb S^4} 
\]
the instanton on $\mathbb S^4$ minimizing the Yang-Mills energy on associated $SU(n)$-bundles $E\to\mathbb S^4$ for the adjoint representation, under fixed second Chern number constraint $c_2(E)=1$:
\begin{equation*}
A\in\op{argmin}\left\{\int_{\mathbb S^4}|F_A|^2d\op{vol}_{\mathbb S^4}\left|\begin{array}{c} A\text{ is loc. }W^{1,2},\\ \frac{1}{8\pi^2}\int_{\mathbb S^4}\op{tr}(F_A\wedge F_A)=1\end{array}\right.\right\}\ .
\end{equation*}
The underlying minimization was studied in \cite{Uhl3}, where it was proved that the minimizer exists and the Chern class constraint is preserved under the underlying weak convergence of connections. It is well known (see \cite{DoKr}, \cite{FreedUh}) that in this case any minimizing curvature must be \textbf{anti-self-dual}. Indeed by Chern-Weil theory we may write
\begin{equation}\label{chernclass}
 c_2(E)=\frac{1}{8\pi^2}\int_{\mathbb S^4}\op{tr}(F_A\wedge F_A)=\frac{1}{8\pi^2}\int_{\mathbb S^4}\left(|F_A^-|^2 -|F_A^+|^2\right)d\op{vol}_{\mathbb S^4}\ .
\end{equation}
Recall that the space of $\mathfrak{su}(n)$-valued $2$-forms $\Omega^2$ splits into the $L^2$-orthogonal eigenspaces $\Omega^\pm$ of the Hodge star operator $\*:\Omega^2\to \Omega^2$ and thus $F_A$ splits as $F_A=F_A^++F_A^-$ with $*F_A^\pm=\pm F_A^\pm$. It follows then from equation \eqref{chernclass} that minimizers of $\mathcal{YM}$ are anti-self-dual and we have
\begin{equation}\label{asdnorm}
 \op{tr}(F\wedge F)=|F|^2\ .
\end{equation}

\subsection{Spaces of weak connections}
In \cite{Uhl2},\cite{Uhl3} the analytic study of Yang-Mills connections on bundles $E\to M^4$ over $4$-dimensional compact Riemannian manifolds $\mathbb M^4$ individuated the following atural space of $\mathfrak{su}(n)$-valued connection forms:
\[
 \mathcal A_{SU(n)}(M^4):=\left\{
\begin{array}{c}
 A\in L^2,\ F_A\stackrel{\mathcal D'}{=}dA+A\wedge A\in L^2\in L^2\ ,\\[3mm]
\text{and loc. }\exists\ g\in W^{1,2}(M^4,SU(n))\text{ s.t. }A^g\in W^{1,2}_{loc}
\end{array}
\right\}\ ,
\]
where $A^g:=g^{-1}dg+g^{-1}Ag$ is the formula representing the change of a connection form $A$ under a change of trivialization $g$.\\

For two $L^2$ connection forms $A,A'$ over $\mathbb B^5$ we write $A\sim A'$ if there exists a gauge change $g\in W^{1,2}(\mathbb B^5,SU(n))$ such that $A'=A^g$. The class of all such $L^2$ connection forms $A'$ is denoted $[A]$. In \cite{PR3} a class suited to the direct minimization of $\mathcal{YM}$ in $5$ dimensions was defined as follows:
\begin{equation*}
 \mathcal A_{SU(n)}(\mathbb B^5):=\left\{
 \begin{array}{c}
  [A]:\:A\in L^2,\ F_A\stackrel{\mathcal D'}{=}dA+A\wedge A\in L^2\\[3mm]
  \forall p\in\mathbb B^5\text{ a.e. }r>0,\:\exists A(r)\in\mathcal A_{SU(n)}(\partial B_r(p))\\[3mm]
  i^*_{\partial B_r(p)}A\sim A(r)
 \end{array}
 \right\}\ .
\end{equation*} 
 Let $\phi$ be a smooth $\mathfrak{su}(n)$-valued connection $1$-form over $\partial \mathbb B^5$. Recall from \cite{PR3} that $A_{SU(n)}(\mathbb B^5)$ is the strong $L^2$-closure of the following space of more regular connections:
 \[
 \mathcal R^{\infty}(\mathbb B^5):=\left\{
 \begin{array}{c}
  F\text{ corresponding to some }[A]\in \mathcal A_{SU(n)}(\mathbb B^5)\text{ s.t. }\\[3mm]
  \exists k,\exists a_1,\ldots,a_k\in \mathbb B^5,\quad F=F_\nabla \text{ for a smooth connection}\nabla\\[3mm]
  \text{on some smooth }SU(n)\text{-bundle }E\to \mathbb B^5\setminus\{a_1,\ldots,a_k\}
 \end{array}
\right\}\ .
\]
 In \cite{PR3} it was proved that the trace condition $i^*_{\partial \mathbb B^5}A\sim\phi$ can be formalized e.g. and the class $\mathcal A_{SU(n)}^\phi(\mathbb B^5)$ of weak connections with trace $\phi$ was introduced. A characterization of such class is as the strong closure of connection classes $[A]\in\mathcal R^\infty(\mathbb B^5)$ which satisfy $i^*_{\partial \mathbb B^5}A\sim\phi$. The main results of \cite{PR3} can be combined into the following theorem:
 \begin{theorem}[Main results of \cite{PR3}]\label{mtpr3}
The minimizer of 
\[
 \inf\left\{\|F_A\|_{L^2(\mathbb B^5)}:\:[A]\in\mathcal A_G^\phi(\mathbb B^5)\right\}
\]
exists and is smooth outside a set of isolated singular points.
 \end{theorem}
A question which arised naturally is whether such result is optimal.\\

Indeed in \cite{Tian} a conjecture was formulated, according to which the singular set of $F$ would have Hausdorff dimension smaller than  $n-6$ in the case of so-called admissible $\Omega$-\textbf{anti-self-dual} curvatures $F$, i.e. curvatures satisfying $F=-\*\Omega\wedge F$ for a smooth closed $(n-4)$-form $\Omega$ in dimension $n$ under the further requirement of $F$ being admissible i.e. that the underlying connection be locally smooth outside of a $(n-4)$-dimensional rectifiable set. A natural question is to ask for examples of stationary or energy-minimizing connection classes which show that the $\Omega$-anti-self-dual requirement is necessary.\\

To the author's knowledge, in the literature no proof is available that stationary curvatures $F$ having a singular set of Hausdorff codimension greater or equal than $5$ exist. This situation is similar to the one taking place in the theory of harmonic maps precedently to R. Hardt, F.H. Lin and C.Y. Wang's celebrated paper \cite{HLW} where it was proved that the map $x/|x|:\mathbb B^3\to\mathbb S^2$ minimizes the $p$-th norm of the gradient among maps whose boundary trace is equal to the identity.\\

\subsection{Main result of the paper}
The main goal of this paper is to show that a result similar in spirit to \cite{HLW} holds for the case of Yang-Mills minimization in dimension $5$.
\begin{theorem}[Main result]\label{mainthm}
 Consider the connection form $A_{\mathbb S^4}$ and its pullback $A_{rad}:=\left(\frac{x}{|x|}\right)^*A_{\mathbb S^4}$. Then $A_{rad}$ is a minimzer for the problem
 \[
 \min\left\{\int_{\mathbb B^5}|F_A|^2d\op{vol}_{\mathbb B^5}:\ [A]\in\mathcal A_{SU(n)}(\mathbb B^5),\ [i^*_{\partial B^5}A]= [A_{\mathbb S^4}]\right\}\ .
 \]
\end{theorem}
Note that for the above minimizer $[A]$ there holds 
\[
 d\left(\op{tr}(F_A\wedge F_A)\right)=8\pi^2\delta_0\ ,
 \]
i.e. the minimizer presents a topological singularity.\\

We note that the curvature form $F_{rad}:=F_{A_{rad}}$ of Theorem \ref{mainthm} is $\Omega$-anti-self-dual with respect to the radial $1$-form
\[
\Omega=\sum_{k=1}^5\frac{x_k}{|x|}dx_k=dr
\]
outside the origin. In other words we have $F_{rad}\wedge \Omega=-*F_{rad}$. The form $\Omega$ is closed in the sense of distributions, however it is not smooth. Therefore it does not fully enter the setting presented in \cite{Tian} and Conjecture 2 of \cite{Tian} remains still open.

\subsection{Minimizing $L^1$ vector fields with defects}
We note that to an $L^2$-integrable $\mathfrak{su}(n)$-values $2$-form $F$ defined on $\mathbb B^5$ we may associate an $L^1$ vector field $X$ by requiring the duality formula
\[
 \langle \phi,X\rangle=\langle \op{tr}(F\wedge F),*\phi\rangle=\int_{\mathbb B^5}\op{tr}(F\wedge F)\wedge \phi
\]
to hold for all smooth $1$-forms on $\overline{\mathbb B^5}$. Through the pointwise inequality 
\begin{equation}\label{i1}
 |\op{tr}(F\wedge F)|\leq |F|^2 
\end{equation}
we also deduce that 
\begin{equation}\label{i2}
 \|X\|_{L^1(\mb B^5)}=\|\op{tr}(F\wedge F)\|_{L^1(\mb B^5)}\leq \|F\|_{L^2(\mb B^5)}^2\ .
\end{equation}
Note that the curvature form $A_{rad}$ of Theorem \ref{mainthm} realizes the pointwise equality in \eqref{i1} and thus also in \eqref{i2}. We will deduce our main theorem from the following result, which is of independent interest.
\begin{theorem}\label{minvf}
 The vector field $X_{rad}(x)=|\mathbb S^4|^{-1}\frac{x}{|x|^5}$ minimizes the $L^1$-norm among vector fields $X\in L^1(\mathbb B^5,\mathbb R^5)$ which are locally smooth outside some finite subset $\Sigma\subset\mathbb B^5$, satisfy
 \begin{equation}\label{divx}
\op{div}X\llcorner\mathbb B^5=\sum_{x\in\Sigma}d_x\delta_x\ ,
 \end{equation}
for some integers $d_x$ and $X\cdot\nu_{\mathbb S^4}\equiv 1$ where $\nu$ is the interior normal vector field to $\mathbb S^4$.
\end{theorem}
This result is proved using similar tools as in \cite{P3}, i.e. Smirnov's decomposition for $1$-currents \cite{smirnov} and a combinatorial result based on the maxflow-mincut theorem.

\textbf{Acknowledgements.} I whish to thank Tristan Rivi\`ere and Gang Tian for interesting discussions from which the question underlying the present work emerged. 

\section{Proof of Theorem \ref{minvf}}
\subsection{Smirnov's decomposition and combinatorial reduction}\label{secsmi}
We start by recalling a version of Smirnov's result \cite{smirnov}, which allows to reduce the larger step of the prof of Theorem \ref{minvf} to a combinatorial argument. We formulate the result in the case of vector fields with divergence of special form as in Theorem \ref{minvf}.\\

We recall the following definitions and notations:
\begin{itemize}
 \item An \textbf{arc} in $\overline{\mathbb B^5}$ is a rectifiable curve which has an injective parameterization $\gamma:[0,1]\to\overline{\mathbb B^5}$. To an arc we may associate a continuous linear functional on smooth $1$-forms $\alpha$ given via $\langle [\gamma],\alpha\rangle:=\int_\gamma\omega$. We also call an arc the functional $[\gamma]$.
 \item The space of all arcs $[\gamma]$ is topologized with the weak topology. Note that the functions $s([\gamma]), e([\gamma])$ which to an arc assign its starting and ending point respectively, are Borel measurable. The variation measure of such $[\gamma]$ is denoted by $\|\gamma\|$ and its total variation is the lenght $\|\gamma\|(\overline{\mathbb B^5})=\ell(\gamma)$. The boundary $\partial[\gamma]$ is given by the difference of Dirac masses $\delta_{e([\gamma])}-\delta_{s([\gamma]}$. and its variation measure is $\delta_{e([\gamma])}+\delta_{s([\gamma]}$.
 \item We say that two vector fields $B,C$ with divergences of finite total variation \textbf{decompose} a vector field $A$ if $A=B+C$, $|A|=|B|+|C|$. We say that a vector field $X$ is \textbf{acyclic} if for each such decomposition with $\partial C=0$ there holds $C=0$. Note that any minimizer as in Theorem \ref{minvf} must be acyclic since if we had a decomposition $X=B+C$ as above with $\partial C=0, C\neq 0$ then $B$ would be a competitor to $X$ of strictly smaller $L^1$ norm.
\end{itemize}

\begin{theorem}[Decomposition of vector fields, \cite{smirnov}]\label{smi}
 Assume $X$ is an acyclic $L^1$ vector field over $\overline{\mathbb B^5}$ such that $\op{div}X$ is a measure of finite total variation. We may then find a finite Borel measure over the space of arcs such that the following hold for all smooth $1$-forms $\alpha$ and for all smooth functions $f$ over $\overline{\mathbb B^5}$:
 \begin{eqnarray}
 \langle X,\alpha\rangle&=&\int \langle[\gamma],\alpha\rangle d\mu(\gamma)\ ,\label{s1}\\
 \langle |X|,f\rangle &=&\int \langle \|\gamma\|,f\rangle d\mu(\gamma)\ ,\label{s2}\\
 \langle \op{div}X,f\rangle&=&\int \langle\delta_{e([\gamma])}-\delta_{s([\gamma]}, f\rangle d\mu(\gamma)\ ,\label{s3}\\
 \langle \|\op{div}X\|, f\rangle&=&\int\langle \delta_{e([\gamma])}+\delta_{s([\gamma]},f\rangle d\mu(\gamma)\ .\label{s4}
 \end{eqnarray}
\end{theorem}
In other words the vector field $X$ decomposes as a superposition of arcs without cancellations whose boundaries decompose $\op{div}X$ without cancellations. See Figure \ref{curves}.

\begin{figure}[h]
\centering{
\resizebox{0.4\textwidth}{!}{\includegraphics{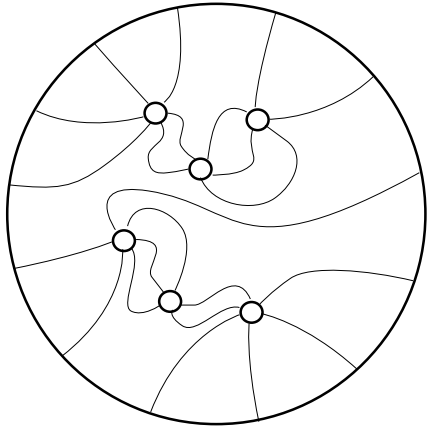}}}
\caption{\emph{We present here some of the arcs in the Smirnov decomposition of a vector field $X$ as in Theorem \ref{minvf}, including some of the charges and some of the curves in the support of $\mu$ as in Theorem \ref{smi}. Note that the arcs are actually oriented.}}\label{curves}
\end{figure}

Note that for a vector field $X$ as in Theorem \ref{minvf} for $\mu$-a.e. arc $\gamma$ the starting point $s([\gamma])$ is in one of the points $x$ such that $d_x<0$ in the expression of $\op{div}X\llcorner\mathbb B^5$ and the end point $e([\gamma])$ is either in one of the points $x$ with $d_x>0$ or on the boundary $\partial B^5$. We may then consider as in \cite{P3} the Borel sets of the form
\begin{equation}\label{posdir}
 C_{x,y}=\{[\gamma]\ \text{ such that } s([\gamma])=x, e([\gamma])=y\} \text{with }d_x<0,d_y>0
\end{equation}
and 
\begin{equation}\label{posdirbdry}
 C_{\partial,y}=\{[\gamma]\ \text{ such that } s([\gamma])\in\partial \mathbb B^5, e([\gamma])=y\},\quad \text{with }d_y>0\ ,
\end{equation}
where $d_x$ are the integers appearing in \eqref{divx}. We observe that the above sets $C_{*,*}$ form a finite partition of $\mu$-almost all of $\op{spt}(\mu)$ and that given a function 
\[
\alpha:\{C_{x,y}\}\cup\{C_{\partial,y}\}\to[-1,1]
\]
the measure 
\begin{equation}\label{muf}
  \mu_\alpha:=\sum_{x,y}\alpha(C_{x,y})\mu\llcorner C_{x,y} + \sum_x\alpha(C_{\partial,y})\mu\llcorner C_{\partial,y}
\end{equation}
also gives an $L^1$ vector field $X_\alpha$ with finite variation divergence which is defined via an equation like \eqref{s1} and saisfies \eqref{s2}, \eqref{s3} but not necessarily \eqref{s4}. The measure $\mu_{X_\alpha}$ obtained applying Theorem \ref{smi} to $X_\alpha$ is supported on curves which are concatenations of curves in the support of $\mu_\alpha$.

\subsection{Combinatorial results}
We now fix the notations for a combinatorial structure which will be associated to data $X, \mu$ given above.
\begin{definition}[$X$-graph]\label{defgr}
 Consider a finite set of vertices $V$ of which a special vertex $\partial\in V$ is highlited, a graph on $V$, i.e. a subset $E\subset V\times V$ such that $(a,b)\in E\Rightarrow (b,a)\in E$, and a weight function $w:E\to\mathbb R$ such that $w(a,b)=-w(b,a)$ unless $a=b=\partial$, in which case we require $w(\partial,\partial)\geq0$. We call such data $(V,E,w,\partial)$ an \textbf{$X$-graph} if
 \begin{enumerate}
  \item We have integer fluxes: $\forall v\in V\setminus\{\partial\},\ fl(v):=\sum_aw(a,v)\in\mathbb Z$, where the sum is taken over all $a\in V$ such that $(a,v)\in E$.
  \item $fl(\partial):=\sum\{w(a,\partial):\ (a,\partial)\in E\setminus(V\times\{\partial\})\}=-1$.
  \item For all $v\in V\setminus\{\partial\}$ we have that all terms in the sum defining $fl(v)$ have the same sign.
 \end{enumerate}
 We denote $V^+$ the vertices for which the sign in the last point is positive and $V^-$ those for which it is negative. Given $x,y\in V$ such that $(x,y), (y,x)\in E$ we say that $(x,y)$ is a \textbf{directed edge} if $w(x,y)>0$.
\end{definition}
We associate an $X$-graph to a vector field $X$ as in Theorem \ref{minvf} by defining
\begin{eqnarray*}
 V&:=&\{x:d_x\neq 0\}\cup\{\partial\}\ ,\\
 E&:=&\{(x,y), (y,x), (\partial,y), (y,\partial):\ C_{x,y}, C_{\partial,y}\text{ are as in \eqref{posdir},\eqref{posdirbdry}}\}
\end{eqnarray*}
and for $*\in V$
\[
  w(*,y):=\mu(C_{*,y}),\ w(y,*):=-\mu(C_{*,y})\quad\text{ if }C_{*,y}\text{ appears in \eqref{posdir} or \eqref{posdirbdry}}\ .
\]
We leave the verification of the properties as in Definition \ref{defgr} to the reader. Note that in this case we obtain property (3) of Definition \ref{defgr} also for $v=\partial$ and we have $\partial\in V^-, w(\partial,\partial)=0$. We also need the following definition which is a modification of that of an $X$-graph.
\begin{definition}[$\bar X$-graph]\label{defbgr}
Consider $(V,E,w, \bar\partial)$ as in Definition \ref{defgr} except that $\bar\partial \subset V$ may now contain more than one vertex, and that we require $w(a,b)=-w(b,a)$ for all edges $(a,b)\in E$ with no exception. We say that $(V,E,w,\bar\partial)$ form an $\bar X$-graph if
\begin{enumerate}
 \item $\forall v\in V\setminus\bar\partial$, $fl(v)\in\mathbb Z$.
 \item $\sum_{v\in\bar\partial} fl(v)=0$.
 \item For $v\in V\setminus\bar \partial$ all terms in the sum defining $fl(v)$ have the same sign.
\end{enumerate}
\end{definition}
Note that as a consequence of the fact that $w(a,b)=-w(b,a)$ even for $a,b\in\bar\partial$ it follows that a $\bar X$-graph has no loops, unlike $X$-graphs who were allowed to have loops on the boundary.

A tool in our combinatorial construction will be the maxflow-mincut theorem, to state which we recall the definition of a (combinatorial) flow.
\begin{definition}[$X$-flows, $\bar X$-flows and cuts]\label{cflow}
 Let $(V,E,w,\partial)$ be an $X$-graph and fix a vertex $a^+\in V^+$. A function $f:E\to\mathbb R$ such that $f(a,b)=-f(b,a)$ for $(a,b)\neq (\partial,\partial)$ and $f(\partial,\partial)\geq 0$ is a \textbf{$X$-flow} if the following properties hold:
 \begin{enumerate}
  \item $|f(a,b)|\leq |w(a,b)|$ for all $(a,b)\in E$. 
  \item For all $v\in V\setminus\{\partial,A^+\}$ there holds $\sum\{f(a,v):\ (a,v)\in E\}=0$.
  \item For all $(x, \partial),(a^+,y)\in E$ there holds $\op{sgn}(f(x,\partial))=\op{sgn}(w(x,\partial))$ and $\op{sgn}(f(a^+,y))=\op{sgn}(w(a^+,y))$.
 \end{enumerate}
We call the vertex $a^+$ the \textbf{sink} of the $X$-flow $f$. The \textbf{value of the $X$-flow} $f$ is by definition the number $val(f):=\sum\{f(\partial,y):\ (\partial, y)\in E, y\neq\partial\}$. We say that the edge $(a,b)$ is \textbf{saturated} by $f$ if equality holds in point 1. above. If all edges with an end equal to $v\in V$ are saturated, we say that $f$ \textbf{saturates $v$}.\\

If $(V, E, w,\bar\partial)$ is a $\bar X$-graph then we define a \textbf{$\bar X$-flow} as above, except that $f(a,b)=-f(b,a)$ will be required to hold for all edges $(a,b)$ with no exception.\\

For a given vertex $a^+\in V^+$ a \textbf{cut between $\partial$ and $a^+$} of the $X$-graph $(V, E,w,\partial)$ is a subset $S\subset E$ such that for every path $(\partial:=a_0,a_1)$, $(a_1,a_2)$, $\ldots,$ $(a_k,a_{k+1}:=a^+)$ such that $(a_i,a_{i+1})\in E$ for all $i=0,\ldots,k$ there exists an index $i$ such that either $(a_i, a_{i+1})\in S$ or $(a_{i+1},a_i)\in S$. The \textbf{ value of a cut $S$} is by definition the number $val(S):=\sum\{|w(x,y)|:\ (x,y)\in S\}$.\\

We say that a $X$-flow (or a $\bar X$-flow) $f$ \textbf{saturates the cut} $S$ if it saturates all edges of $S$.
\end{definition}
Note that for a $X$-flow the following are equivalent: (a) $f$ saturates $\partial$; (b) $f=w$ on all edges with an end in $\partial$; (c) $f$ has value $1$.
Our main combinatorial result is the following:
\begin{proposition}[existence of a saturating $X$-flow]\label{satfl}
 Let $(V, E, w,\partial)$ be an $X$-graph. Then we may find a $X$-flow $f$ which saturates $\partial$.
\end{proposition}
We will need the following result present in \cite{P3}, of which we present a different proof:
\begin{proposition}[existence of a saturating $\bar X$-flow, \cite{P3}]\label{satbfl}
 Let $(V'', E'', w'',\bar\partial)$ be a $\bar X$-graph with the bound
 \begin{equation}\label{smallflux}
  \sum\left\{|w''(a,b)|:\ a\in\bar\partial,\ (a,b)\in E\right\}<1\ .
 \end{equation}
Then there exists a $\bar X$-flow $f''$ saturating $\bar\partial$.
\end{proposition}
\begin{proof}[Proof of Proposition \ref{satbfl}:]
 We proceed by induction on the number of non-boundary vertices $\#(V''\setminus\bar\partial)$. In the case $\#(V''\setminus\bar\partial)=0$ we may take $f''=w''$ and we cnclude.\\

 Supposing that the statement is true when $\#(V''\setminus\bar\partial)<n$ we may prove it fo the case $\#(V''\setminus\bar\partial)=n$ as follows.\\

 Up to reducing to the connected components we may assume that the underlying graph $(V'', E'')$ is connected.\\

 Applying the maxflow-mincut theorem \cite{ff} we obtain the existence of a $\bar X$-flow of maximum value $f$ and of a minimum value cut $S$ saturated by $f$. Then $S$ separates the graph $(V'', E'')$ into two connected components $V_1\supset\bar\partial^+, V_2\supset\bar\partial^-$ which are both $\bar X$-graphs with $\#(V\setminus \bar\partial)<n$, if we take $V_i':=V_i\cup\{\text{ends of edges in }S\}$, $E_i:=[E''\cap(V_i\times V_i)]\cup S$, $w_i:=w''|_{E_i}$ for $i=1,2$
\[
 \bar\partial_1=\bar\partial^+\cup\{\text{ends of edges in }S\}\cap V_2,\quad \bar\partial_2=\bar\partial^-\cup\{\text{ends of edges in }S\}\cap V_1\ .
\]
The properties of a $\bar X$-graph are all easy to verify for the $(V_i, E_i,w,\bar\partial_i)$ except perhaps for property 2. in Definition \ref{defbgr}, i.e. the fact that the total flux through the boundaries are zero. To prove this we use the bound \eqref{smallflux} and the integrality condition in the definition of a $\bar X$-graph. Let $S\pm$ be the set of edges in $S$ for which $f''(a,b)=\pm w''(a,b)$. We have $val(f)<1/2$ as a consequence of the fact that in the original $\bar X$-graph the total flux through $\bar\partial$ was zero and of \eqref{smallflux}. Since $f''$ saturates $S$ we have
\[
1/2>val(S)=\sum_{(a,b)\in S}|w''(a,b)|=val(S^+)+val(S^-)\ .
\]
From the integrality condition 1. in Definition \ref{defbgr} and from the fact that $S$ disconnects the underlying graph $(V'', E'')$ we deduce that the total flux through $\bar\partial_i$ must be an integer for $i=1,2$. Ath the same time the absolute value of this flux is bounded by $val(\bar\partial^+)+val(S)<1$. Therefore it must be zero, and condition 2. in Definition \ref{defbgr} is verified.\\

We may then apply the inductive hypotheses and obtain flows $f_1,f_2$ saturating $\bar\partial_1,\bar\partial_2$. In particular these flows coincide on $S$ and they extend to $\bar X$-flow $f''$ over the initial $\bar X$-graph. The fact that $f''$ saturates $\bar\partial$ follows from the fact that $f_i$ saturates $\bar\partial_i$ for $i=1,2$.
\end{proof}

\begin{proof}[Proof of Proposition \ref{satfl}:]
 We proceed by induction on the number of vertices $\#V$. For $\#V=2$ it suffices to take $f=w$.\\
 
Let now $n>2$, assume that the thesis is true whenever $\#V<n$, and consider the case $\#V=n$. Choose as a sink a vertex $a^+\in V^+$. Consider a $X$-flow $f$ with sink $a^+$ of maximal value. By the maxflow-mincut theorem \cite{ff} there exists a cut $S$ realizing the minimum of possible values of cuts between $\partial$ and $a^+$ and such that $f$ saturates $S$ and the value of $f$ equals the value of $S$.\\
 
 If $val(f)=1$ then $f$ saturates $\partial$ and we conclude. If $val(f)=0$ then $S$ can be taken to be empty, thus $a^+, \partial$ are in different connected components of the $X$-graph. In this case we may remove the connected component of $a^+$ and reduce to the case $\#V<n$. We then conclude by inductive hypothesis.\\
 
 Consider now the remeining case when the value of $f$ is in $]0,1[$. Let $S^\pm$ be the sets of edges in $S$ for which $\op{sgn}(w)=\pm$ and let $s^\pm:=\sum_{S^\pm}|w(e)|$. We then conclude from the definition of an $X$-graph and from the fact that $f$ saturates $S$ that
 \begin{equation}\label{zeroflow}
  s^+-s^-\in\mathbb Z,\quad val(f)=s^++s^-\in]0,1[\ \text{ thus }s^+=s^-\ .
 \end{equation}
We then replace the $X$-graph $(V,E,\partial,w)$ with the $X$-graph $(V',E',\partial,w')$ defined as follows:
\begin{itemize}
 \item $V'\subset V$ consists of all vertices in the connected component of $\partial$ with respect to the cut $S$.
 \item $E'$ consists of all edges in $E\cap (V'\times V')$ and of new edges of the form $(v',\partial), (\partial,v')$ where $(v',x)\in S,v'\in V'$.
 \item $w'$ is defined to be equal to $w$ on $E\cap V'\times V'$, while for $(v',x)\in S,v'\in V'$ we define $w'(v',\partial)=-w'(\partial,v')=w(v',x)$.  
\end{itemize}
We see that the properties as in the definition of an $X$-graph are trivially valid at vertices $v'\in V'\setminus\{\partial\}$ while at $\partial$ they are still valid due to \eqref{zeroflow}.\\

Since $\#V'<\#V$ we may apply the inductive hypothesis to $(V',E',\partial,w')$ and find a $X$-flow $f'$ with sink $b^+\in (V')^+$ which saturates $\partial$. We then extend it to a $X$-flow $\bar f$ on the original $X$-graph $(V, E, \partial,w)$ as follows.\\

Note that we may define a $\bar X$-graph $(V'',E'', w'', \bar\partial)$ by defining
\begin{eqnarray*}
 V''&:=&(V\setminus V')\cup\{\text{ends of edges in }S\}\ ,\\
 E''&:=&\left[E\cap(V''\times V'')\right]\cup S\ ,\\
 w''&:=&w|_{E''}\ ,\\
 \bar \partial&:=&\{\text{ends of edges in }S\}\setminus V'\ .
\end{eqnarray*}
By applying Proposition \ref{satbfl} to this $\bar X$-graph we may find a flow $f''$ on it saturating $\bar\partial$. In particular this flow conicides with $f'$ on $\bar\partial$ and the extension $\bar f$ of $f'$ via $f''$ is well-defined and is an $X$-flow, saturating $\partial$, as desired.
\end{proof}
See Figure \ref{cutfig} for the corresponding picture for vector fields $X$ as in Theorem \ref{minvf}.

 \begin{figure}[h]
\centering{
\resizebox{0.6\textwidth}{!}{\includegraphics{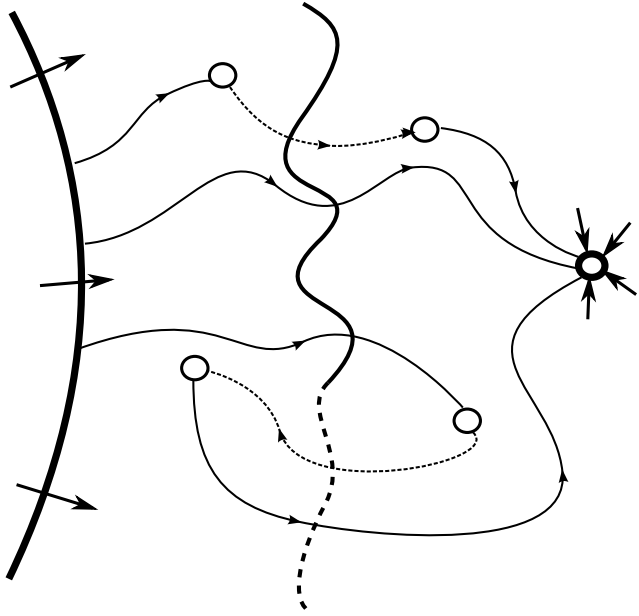}}}
\caption{\emph{We represent a possible situation in the last step of the proof of Proposition \ref{satfl}, seen at the level of the vector fields $X$ as in Theorem \ref{minvf}. On the left we represent the boundary $\partial \mathbb B^5$ and on the right we have a positive charge corresponding to a possible sink of the saturating flow of Proposition \ref{satfl}. Some arcs corresponding to the saturating flow are represented by directed lines. These arcs are obtained as concatenations of arcs corresponding to the original vector field $X$, with the same orientation (continuous parts) or reversed orientation (dashed parts). The continuous vertical wiggly line in the center of the figure represents a minimal cut. We enlarging the cut with the dashed part of that line would make it non-minimal.}}\label{cutfig}
\end{figure}

\subsection{Proof of Theorem \ref{minvf}}
\begin{proof}
We apply Proposition \ref{satfl} to the $X$-graph associated to a vector field $X$ as in Theorem \ref{minvf} and to the associated measure on arcs $\mu$ given by Theorem \ref{smi} applied to $X$.\\

From the flow $f$ as given in Proposition \ref{satfl} if $\alpha=\chi_{\{w\neq 0\}}f/w$ we construct a measure $\mu_\alpha$ as in \eqref{muf}. This measure on arcs gives an $L^1$ vector field $X_\alpha$ which in turn decomposes via Smirnov's Theorem \ref{smi} via a different measure $\mu'$, this time into a superposition of arcs $\gamma$ such that
\begin{eqnarray*}
 s([\gamma])&\in&\partial \mathbb B^5\ ,\\
 e([\gamma])&=&a^+\ ,
\end{eqnarray*}
where the point $a^+$ is the one where the charge corresponding to the sink of the flow $f$ is located.\\

Note that from the remarks at the beginning of Section \ref{secsmi} the lenghts of such arcs give the $L^1$ norm of the vector field $X_\alpha$ and we have
\begin{equation}\label{1}
 \int_{\mathbb B^5}|X_\alpha|=\int\ell(\gamma)d\mu'(\gamma)\le \int\ell(\gamma) d\mu(\gamma)=\int|X|.
\end{equation}
The inequality is due to the fact that since we only decreased the weights of curves the new vector field $X_\alpha$ satisfies pointwise a.e. $x$ the inequality $|X_\alpha|(x) \le |X|(x)$. Note however that while $\op{div}(X_\alpha)=\delta_{a^+}+|\mathbb S^4|^{-1}\mathcal H^4\ll\mathbb S^4$ remains valid in the sense of distributions, we don't have anymore the information that $X_\alpha$ is locally smooth on $\mathbb B^5\setminus\{a^+\}$.\\

For $\mu'$-a.e. $\gamma\in\op{spt}(\mu')$ we have
\[
 \ell(\gamma)\ge|s([\gamma]) - e([\gamma])|\ ,
\]
thus if we denote the arc corresponding to a segment from $a$ to $b$ by $[\gamma]=[a,b]$ then we may write
\begin{eqnarray}
 \int_{\mathbb B^5}|X_\alpha|&=&\int \ell(\gamma)d\mu'(\gamma)\nonumber\\
 &\ge&\int |s([\gamma]) - e([\gamma])|d\mu'(\gamma)\nonumber\\
 &=&\int_{\mathbb S^4}|x-a^+|d\mathcal H^4(x)\ .\label{2}
\end{eqnarray}
Note that
\[
 \left.\frac{\partial}{\partial t}\right|_{t=0}\int_{\mathbb S^4}|x-(a^+ +tb)|d\mathcal H^4(x)=-b\cdot\int_{\mathbb S^4}\frac{x-a^+}{|x-a^+|}d\mathcal H^4(x)\ ,
\]
thus the minimum of \eqref{2} is achieved when $a^+$ satisfies
\[
 \fint_{\mathbb S^4}\frac{x-a^+}{|x-a^+|}=0\ ,
\]
i.e. precisely for $a^+=0$. The thesis of Theorem \ref{minvf} now follows from \eqref{1}, \eqref{2}.
\end{proof}

\section{Proof of Theorem \ref{mainthm}}
\begin{proof}
 We see from \eqref{i1}, \eqref{i2} that for the Poinca\`E-Hodge dual vector field $X$ to $\op{tr}(F\wedge F)$ there holds $\|X\|_{L^1}\leq\|F\|_{L^2}^2$ with equality in the case of the curvature form $F_{rad}$ as in Theorem \ref{mainthm}. From the fact \cite{PR3} that $\mathcal R^\infty(\mathbb B^5)$ is dense in $\mathcal A_{SU(n)}(\mathbb B^5)$ it follows that the infimum of the $L^2$ norm of the curvature given by Theorem \ref{mtpr3} in the case of the boundary trace $A_{\mathbb S^4}$, is equal to the infimum of the same functional over $[A]\in\mathbb R^\infty(\mathbb B^5)$ with trace $\sim A_{\mathbb S^4}$. In particular we have that
 \begin{eqnarray*}
  \int_{\mathbb B^5}|F_{rad}|^2&=&\int_{\mathbb B^5}|X_{rad}|\\
  &=&\inf\left\{\int |X|:\ X\text{ as in Theorem \ref{minvf}}\right\}\\
  &\leq&\inf\left\{\int |F_A|^2:\ [A]\in\mathcal R^\infty(\mathbb B^5),\ i^*_{\partial \mathbb B^5}A\sim A_{\mathbb S^4}\right\}\\
  &=&\min\left\{\int |F_A|^2:\ [A]\in\mathcal A_{SU(n)}^{A_\mathbb S^4}(\mathbb B^5)\right\}\ .
 \end{eqnarray*}
In particular all inequalities above must be equalities and Theorem \ref{mainthm} follows.
\end{proof}

\end{document}